 \documentclass[12pt,twoside]{amsart}

 \usepackage{graphicx}
 \newtheorem{theorem}{\bf Theorem}
 \newtheorem{definition}{\bf Definition}
 \newtheorem{note}{\bf Note}

 \newtheorem{property}{\bf Property}
 \newtheorem{remark}{\bf Remark}

\begin{document}
\title{A Note on Fuzzy Real and Complex Field}
\author{T. K. Samanta }

 \maketitle

\textit{Department of Mathematics, Uluberia College, Uluberia,
Howrah - 711315, India. \\ E-mail : mumpu$_{-}$tapas5@yahoo.co.in;}
\bigskip
\medskip

\begin{abstract} Using the concept of fuzzy field, we have considered
the fuzzy field of real and complex numbers and thereafter we have
established a few standard results of real and complex numbers with
respect to a membership function.
\end{abstract}

\\\\

\textbf{ Introduction : \;}  The concepts of Fuzzy real and complex
numbers; fuzzy functions; limit, continuity, differentiability,
integrability of fuzzy functions; sequence of fuzzy numbers and
their convergence has been developed in different papers $[ \,1
\,-\, 5 \,]$. All these concepts has been developed without
considering fuzzy field of real or complex numbers. But, in crisp
system, considering the set of real or complex numbers as the field
of real or complex numbers, usual mathematical analysis has been
developed. In the paper $[\,2\,]$, S. Nanda, introduced the concept
of fuzzy field. Using this concept of fuzzy field, we have
considered the fuzzy field of real and complex numbers and
thereafter we have established a few standard results of real and
complex numbers with respect to a membership function. In
particular, if we consider the characteristic function instead of
membership function, our results will coincide with the usual
results of real and complex numbers. In the first section, we have
introduced the fuzzy field of real numbers \, $\mathbb{R_{ \,\mu} }$
, modulus of a member of \, $\mathbb{R_{ \,\mu} }$ , supremum ,
infimum of a subset of \, $\mathbb{R_{ \,\mu} }$ and a few
properties depending on these concepts. In section two, we have
introduced the concepts of sequence in \, $\mathbb{R_{ \,\mu} }$ ,
convergence and divergence of a sequence in \, $\mathbb{R_{ \,\mu}
}$. In section three , we have introduced the concept of fuzzy field
of complex numbers \, $\mathbb{C_{ \,\mu} }$ , complex conjugate,
modulus, argument of a member of \, $\mathbb{C_{ \,\mu} }$ ,
exponential function, logarithm function in \, $\mathbb{C_{ \,\mu}
}$ and a few properties depending on these concepts. Lastly we have
given a conclusion, where we have described our future research.
\bigskip
\bigskip
\begin{definition}
$[ \,2\, ] $ Let \, $X$ \, be a field and \, $F$ \, a fuzzy set in
\, $X$ \, with membership function \, $\mu_{\,F}$. \, Then \, $F$ \,
is a fuzzy field in \, $X$ \, iff. the following conditions are
satisfied $:$ \[ ( i ) \hspace{2.5cm} \mu_{\,F}\,(\, x \; + \; y \,)
\; \; \geq \; \; \min \, \{ \, \mu_{\,F}( \, x \, ) \; , \;
\mu_{\,F}( \, y \, ) \} \hspace{1.5cm} for \; \; all  \; \; x \; ,
\; y \; \, \varepsilon \; \, X \] \[ ( ii ) \hspace{2.5cm}
\mu_{\,F}\,(\, - \,x \,) \; \; \geq \; \; \mu_{\,F}( \, x \, )
\hspace{5.5cm} for \; \; all \; \; x \; \, \varepsilon \; \, X \] \[
( iii ) \hspace{2.5cm} \mu_{\,F}\,(\, x \; y \,) \; \; \geq \; \;
\min \, \{ \, \mu_{\,F}( \, x \, ) \; , \; \mu_{\,F}( \, y \, ) \}
\hspace{2.5cm} for \; \; all  \; \; x \; , \; y \; \, \varepsilon \;
\, X \] \[ ( iv ) \hspace{2.5cm} \mu_{\,F}\,(\, x^{\, - \,1} \,) \;
\; \geq \; \; \mu_{\,F}( \, x \, ) \hspace{5.5cm} for \; \; all \;
\; x \; \, \varepsilon \; \, X \] \[ ( v ) \hspace{2.5cm}
\mu_{\,F}\,(\,0\,) \; \; = \; \; 1 \; \; = \; \; \mu_{\,F}\,(\,1\,)
 \hspace{7.5cm} \]
\end{definition}
\medskip
Throughout our discussion , we will consider the field \, $X$ \, as
the field of real numbers \, $\mathbb{R}$ \, or as the field of
complex numbers \, $\mathbb{C}$ . Also, we consider \, $\mathbb{R_{
\,\mu} }$ \, as a fuzzy field of real numbers and \, $\mathbb{C_{
\,\mu} }$ \, as a fuzzy field of complex numbers and the members are
respectively called \, $\mu -$ fuzzy real numbers and $\mu -$ fuzzy
complex numbers.
\medskip
\begin{definition}
For \, $x_{\,1} \; , \; x_{\,2} \; \; \varepsilon \; \; \mathbb{R_{
\,\mu} }$ , \, $x_{\,1} \; \; \leq_{\,\mu} \; \; x_{\,2}$ \, if and
only if \, $x_{\,1} \, \mu\,(\,x_{\,1}\,) \; \; \leq \; \; x_{\,2}
\, \mu\,(\,x_{\,2}\,)$ . \, Similarly , we can define \, $x_{\,1} \;
\; <_{\,\mu} \; \; x_{\,2}$ , \, $x_{\,2} \; \; \geq_{\,\mu} \; \;
x_{\,1}$ , \, $x_{\,2} \; \; >_{\,\mu} \; \; x_{\,1}$ .
\end{definition}

\medskip
\begin{property}
Let \, $a \; , \; b \; \; \varepsilon \; \; \mathbb{R_{ \,\mu} }$ .
Then \, $( \, i \, ) \hspace{0.5cm} a \; \, \geq \; \, 0 \; \;
\Longrightarrow \; \; a \; \, \geq_{\,\mu} \; \, 0$ \, and \, $a \;
\, \leq \; \, 0 \; \; \Longrightarrow \; \; a \; \, \leq_{\,\mu} \;
\, 0$ , \; $( \, ii \, ) \hspace{0.5cm} 0 \; \, \leq_{\,\mu} \; \,
a\; \; \Longrightarrow \; \; - \, a \; \, \leq_{\,\mu} \; \, 0$ , \;
$( \, iii \, ) \hspace{0.5cm} 0 \; \, \leq_{\,\mu} \; \, a \; , \; 0
\; \, \leq_{\,\mu} \; \, b \; \; \Longrightarrow \; \;  0 \; \,
\leq_{\,\mu} \; \, a \; + \; b$ , \; $( \, iv \, ) \hspace{0.5cm} a
\; \, \leq_{\,\mu} \; \, 0 \; , \; b \; \, \leq_{\,\mu} \; \, 0 \;
\; \Longrightarrow \; \;  a \; + \; b \; \, \leq_{\,\mu} \; \, 0$ ,
\; $( \, v \, ) \hspace{0.5cm} 0 \; \, \leq_{\,\mu} \; \, a \; , \;
0 \; \, \leq_{\,\mu} \; \, b \; \; \Longrightarrow \; \;  0 \; \,
\leq_{\,\mu} \; \, a \, b$ , \; $( \, vi \, ) \hspace{0.5cm} a \; \,
\leq_{\,\mu} \; \, 0 \; , \; b \; \, \leq_{\,\mu} \; \, 0 \; \;
\Longrightarrow \; \;  0 \; \, \leq_{\,\mu} \; \, a \, b$ , \; $( \,
vii \, ) \hspace{0.5cm} 0 \; \, \leq_{\,\mu} \; \, a \; , \; b \; \,
\leq_{\,\mu} \; \, 0 \; \; \Longrightarrow \; \;  a \, b \; \,
\leq_{\,\mu} \; \, 0$ , \; $( \, viii \, ) \hspace{0.5cm} a \; \,
\neq \; \, 0 \;  \; \Longrightarrow \; \; 0 \; \, \leq_{\,\mu} \; \,
a^{\,2} $ .
\end{property}
\medskip
\begin{definition}
Let \, $a \; \; \varepsilon \; \; \mathbb{R_{ \,\mu} }$ . Then the
modulus of \, $a$ \, in \, $\mathbb{R_{ \,\mu} }$ \, is denoted by
\, $| \, a \, |_{\,\mu}$ \, and defined as \, $| \, a \, |_{\,\mu}
\; \, = \; \, | \, a \, | \, \mu\,(\,a\,) $ , \, where \, $| \, a \,
|$ \, denotes the usual modulus of \, $a \; \; \varepsilon \; \;
\mathbb{R }$ .
\end{definition}
\medskip
\begin{property}
$( \, i \, ) \hspace{1.5cm} | \, a \, |_{\,\mu} \; \, = \; \,
\left\{ {\,\,\begin{array}{*{20}c}
   { a\,\,\mu \,(\,a\,)\,\hspace{1.6cm} {\text{if      }} \hspace{0.5cm}a\;\, > \;\,0}  \\
   { 0\,\hspace{2.5cm} {\text{if      }} \hspace{0.5cm} a\;\, = \;\,0}  \\
   { - \,a\,\,\mu \,(\,a\,)\, \hspace{1.0cm} {\text{if}} \hspace{0.6cm}a\;\; < \;\;0}  \\

 \end{array} } \right.$

\end{property}

\begin{property}
$(\,ii\,) \hspace{0.5cm} |\, - a \,|_{\,\mu} \; \, = \; \,
|\,a\,|_{\,\mu}$ , \; $(\,iii\,) \hspace{0.5cm} \frac{|\,a \, b
\,|_{\, \mu}}{\mu (\,a \, b\,)} \; \, = \; \, \frac{|\,a \,|_{\,
\mu}}{\mu (\,a\,)} \; \frac{|\,b \,|_{\, \mu}}{\mu (\,b\,)}
\hspace{0.7cm} provided \; \; \mu (\,a\,) , \\ \mu (\,b\,) \,,\, \mu
(\,a\,b\,) \; \, \neq \; \, 0$ , \; $(\,iv\,) \hspace{0.5cm}
\frac{|\,a \;+\; b \,|_{\, \mu}}{\mu (\,a \;+\; b\,)} \; \, \leq \;
\, \frac{|\,a \,|_{\, \mu}}{\mu (\,a\,)} \;+\; \frac{|\,b \,|_{\,
\mu}}{\mu (\,b\,)} \hspace{0.7cm} provided \; \; \mu (\,a\,) \,,\,
\mu (\,b\,) \,,\, \mu (\,a \;+\; b\,) \; \, \neq \; \, 0$ , \;
$(\,v\,) \hspace{0.5cm} Let \, c \; > \; 0 , \; |\,a\,|_{\,\mu} \;
\, < \; \, c \; \; \Longrightarrow \; \; - \; \frac{c}{\mu (\,a\,)}
\; \, < \; \, a \; \, < \; \, \frac{c}{\mu (\,a\,)} \; \; , \; \;
\mu (\,a\,) \; \, \neq \;\, 0$ \; and \; $ |\,a\,| \; \, <_{\,\mu}
\; \, c \; \; \Longrightarrow \; \; - \; \frac{c \,\mu (\,c\,)}{\mu
(\,a\,)} \; \, < \; \, a \; \, < \; \, \frac{c\,\mu (\,c\,)}{\mu
(\,a\,)} \; \; , \; \; \mu (\,a\,) \; \, \neq \;\, 0$
\end{property}
\medskip
\begin{definition}
Let \, $A \; \, \subseteq \; \, \mathbb{R_{\,\mu}}$. \, $A$ is said
to be $\mu\,-$ bounded above in $\mathbb{R_{\,\mu}}$ if the set \,
$\{ \, x \, \mu (\,x\,) \; \, : \; \, x \; \, \varepsilon \; \,A \}$
\, is bounded above in \, $\mathbb{R}$. The \, $\mu\,-$ supremum of
$A$ is denoted by $sup_{\,\mu}\,A$ and defined by \, $sup_{\,\mu}\,A
\; \, = \; \, sup\;\{\, x\,\mu (\,x\,) \; \, : \; \, x \;\,
\varepsilon \;\, A \,\}$ .\\\\ Similarly , $A$ is said to be
$\mu\,-$ bounded below in $\mathbb{R_{\,\mu}}$ if the set \, $\{ \,
x \, \mu (\,x\,) \; \, : \; \, x \; \, \varepsilon \; \,A \}$ \, is
bounded below in \, $\mathbb{R}$. The \, $\mu\,-$ infimum of $A$ is
denoted by $inf_{\,\mu}\,A$ and defined by \, $inf_{\,\mu}\,A \; \,
= \; \, inf\;\{\, x\,\mu (\,x\,) \; \, : \; \, x \;\, \varepsilon
\;\, A \,\}$ . \\\\ A set $A \; \, \subseteq \; \,
\mathbb{R_{\,\mu}}$ is said to be $\mu\,-$ bounded if it is both
$\mu\,-$ bounded above and $\mu\,-$ bounded below.
\end{definition}
\medskip
\begin{theorem}
If $A \;\, \subseteq \;\, \mathbb{R}$ is bounded then it is $\mu\,-$
bounded.
\end{theorem}
\medskip
\begin{proof}
Obvious.
\end{proof}

\textbf{Completeness property of} $\mathbb{R_{\,\mu}} \; :$ \\\\
\textbf{Statement I} $:$ Every non - empty \, $\mu\,-$ bounded above
subset of \, $\mathbb{R_{\,\mu}}$ has a \, $\mu\,-$ supremum. \\\\
\textbf{Statement II} $:$ Every non - empty \, $\mu\,-$ bounded
below subset of \, $\mathbb{R_{\,\mu}}$ has a \, $\mu\,-$ infimum.
\begin{theorem}
Suppose the \textbf{Statement I} holds and $A$ is a non $-$ empty
subset of $\mathbb{R_{\,\mu}}$, which is bounded below in
$\mathbb{R_{\,\mu}}$. Then $A$ has an $\mu\,-$ infimum .
\end{theorem}
\medskip
\begin{theorem}
Let $A$ be a non $-$ empty subset of $\mathbb{R_{\,\mu}}$ and
$\mu\,-$ bounded above in $\mathbb{R_{\,\mu}}$. An $\mu\,-$ upper
bounded $M$ of $A$ is $\mu\,-$ supremum of $A$ if and only if for
each \, $\varepsilon \;\,>\;\, 0$ there exists an element \,
$x_{\,1} \;\,\varepsilon\;\, A$ such that \, $M \;-\; \varepsilon
\;\,<\;\, x_{\,1}\,\mu (\,x_{\,1}\,) \;\,\leq\;\, M$.
\end{theorem}
\medskip
\S \, 2. \hspace{1.0cm} \textbf{Sequence in} $\mathbb{R_{\,\mu}}$ \,
:
\medskip
\begin{definition}
Let \, $\{\,x_{\,n}\,\}_{\,n}$ be a sequence in
$\mathbb{R_{\,\mu}}$. Then \, $\{\,x_{\,n}\,\}_{\,n}$ is said to
converge to $x_{\,0}$ in $\mathbb{R_{\,\mu}}$ if for every \,
$\varepsilon \;\,>\;\, 0$, there exists a natural number $k$ such
that $|\,x_{\,n} \;-\; x_{\,0}\,|_{\,\mu} \;\,<\;\, \varepsilon$
\;\; for all \; $n \;\,\geq\;\, k$. If $\{\,x_{\,n}\,\}_{\,n}$
converges to $x_{\,0}$ in $\mathbb{R_{\,\mu}}$ then we express it by
$x_{\,n} \;\,\mathop  \to \limits_\mu  \;\,x_{\,0} $ \, as \, $n
\;\,\longrightarrow\;\, \infty$ , $x_{\,0}$ is called $\mu\,-$ limit
of the sequence $\{\,x_{\,n}\,\}_{\,n}$ in $\mathbb{R_{\,\mu}}$ and
$\{\,x_{\,n}\,\}_{\,n}$ is called $\mu\,-$ convergent in
$\mathbb{R_{\,\mu}}$.
\end{definition}
\begin{theorem}
If $\{\,x_{\,n}\,\}_{\,n}$ converges to $l$ in $\mathbb{R}$ then
$\{\,x_{\,n}\,\}_{\,n}$ converges to $l$ in $\mathbb{R_{\,\mu}}$.
\end{theorem}
\begin{proof}
Obvious
\end{proof}
\begin{remark}
The converse of the above theorem is not necessarily true. In fact,
$\mu\,-$ limit of $\{\,x_{\,n}\,\}_{\,n}$ is not necessarily unique.
For example , let \, $x_{\,n} \;=\; log n \;+\; 1$ \, and \,
$y_{\,n} \;=\; log n \;+\; \sqrt{\,2}$ \; for all \; $n \;\geq\; 1$
, \, $\mu (\,x_{\,n}\,) \;\,=\;\, \mu (\,y_{\,n}\,) \;\,=\;\,
\frac{n}{(\,n \;+\; 1\,)^{\,3}}$, $\mu (\,x_{\,n} \;+\; x_{\,m}\,)
\;\,=\;\, \mu (\,x_{\,n} \;-\; x_{\,m}\,) \;\,=\;\, \mu (\,x_{\,n}
\; x_{\,m}\,) \;\,=\;\, \mu (\,x_{\,n} \; x_{\,m}^{\,- \,1}\,)
\;\,=\;\, 1$, $\mu (\,y_{\,n} \;+\; y_{\,m}\,) \;\,=\;\, \mu
(\,y_{\,n} \;-\; y_{\,m}\,) \;\,=\;\, \mu (\,y_{\,n} \; y_{\,m}\,)
\;\,=\;\, \mu (\,y_{\,n} \; y_{\,m}^{\,- \,1}\,) \;\,=\;\, 1$ , $\mu
(\,x_{\,n} \;+\; y_{\,m}\,) \;\,=\;\, \mu (\,x_{\,n} \;-\;
y_{\,m}\,) \;\,=\;\, \mu (\,x_{\,n} \; y_{\,m}^{\,-\,1}\,) \;\,=\;\,
\mu (\,x_{\,n}^{\,-\,1} \; y_{\,m}^{\,- \,1}\,) \;\,=\;\, 1$, \\$\mu
(\,x\,) \;=\; 0$ \; for all other \, $x \;\,\varepsilon\;\,
\mathbb{R}$. \; Here we see that $x_{\,n} \;\,\mathop  \to
\limits_\mu  \;\,0 $ \, as \, $n \;\,\longrightarrow\;\, \infty$ but
$\{\,x_{\,n}\,\}_{\,n}$ is not convergent in $\mathbb{R}$. We also
see that $x_{\,n} \;\,\mathop  \to \limits_\mu  \;\,1 \;-\;
\sqrt{\,2} $ \, as \, $n \;\,\longrightarrow\;\, \infty$ .
\end{remark}
\medskip
\begin{remark}
If $inf\,\{\,\mu (\,x\,) \;\,: \;\, x \;\,\varepsilon\;\, \mathbb{R}
\,\}$ \; $= \;\; a \;\,>\;\, 0$ then $\{\,x_{\,n}\,\}_{\,n}$
converges to $l$ in $\mathbb{R}$ if and only if
$\{\,x_{\,n}\,\}_{\,n}$ converges to $l$ in $\mathbb{R_{\,\mu}}$.
\end{remark}
\medskip
\begin{remark}
If $inf\,\{\,\mu (\,x\,) \;\,: \;\, x \;\,\varepsilon\;\, \mathbb{R}
\,\}$ \; $= \;\,a \;\,>\;\, 0$ \, and \, $x_{\,n} \;\,\mathop  \to
\limits_\mu \;\,l $ \, as \, $n \;\,\longrightarrow\;\, \infty$ \,
then \, $l$ is unique .
\end{remark}
\begin{theorem}
If $inf\,\{\,\mu (\,x\,) \;\,: \;\, x \;\,\varepsilon\;\, \mathbb{R}
\,\}$ \; $= \;\,a \;\,>\;\, 0$ \, and \, $\{\,x_{\,n}\,\}_{\,n}$ \,
is $\mu\,-$ convergent in $\mathbb{R_{\,\mu}}$ then \,
$\{\,x_{\,n}\,\}_{\,n}$ \, is $\mu\,-$ bounded in
$\mathbb{R_{\,\mu}}$.
\end{theorem}
\begin{proof}
Let \, $x_{\,n} \;\,\mathop  \to \limits_\mu \;\,l $ \, as \, $n
\;\,\longrightarrow\;\, \infty$. \; $\Longrightarrow$ \; For \,
$\varepsilon \;=\;1$, there exists \, $k \;\,\varepsilon\;\,
\mathbb{N}$ \, such that $|\,x_{\,n} \;-\; l\,|_{\,\mu} \;\,<\;\, 1$
\;\; for all \; $n \;\,\geq\;\, k$ \; $\Longrightarrow$ \; $l\;\mu
(\,x_{\,n} \;-\; l\,) \;-\;1 \;\,<\;\, x_{\,n}\;\mu (\,x_{\,n} \;-\;
l\,) \;\,<\;\, l\;\mu (\,x_{\,n} \;-\; l\,) \;+\; 1$ \;\; for all \;
$n \;\,\geq\;\, k$ \; $\Longrightarrow$ \; $a\,(\,l\,a \;-\; 1\,)
\;\,<\;\, x_{\,n}\;\mu (\,x_{\,n}\,) \;\,<\;\, \frac{1}{a}\;(\,l
\;+\; 1\,)$ \;\; for all \; $n \;\,\geq\;\, k$ \; $\Longrightarrow$
\; $\{ \, x_{\,n}\;\mu (\,x_{\,n}\,) \; \, : \; \, n \; \,
\varepsilon \; \,\mathbb{N} \}$ \, is bounded in $\mathbb{R}$ \;
$\Longrightarrow$ \; $\{\,x_{\,n}\,\}_{\,n}$ \, is $\mu\,-$ bounded
in $\mathbb{R_{\,\mu}}$.
\end{proof}
\begin{remark}
If we drop the condition \, $inf\,\{\,\mu (\,x\,) \;\,: \;\, x
\;\,\varepsilon\;\, \mathbb{R} \,\}$ \; $= \;\,a \;\,>\;\, 0$ \,
from the above theorem , the theorem may fails to hold. For example
, consider \, $x_{\,n} \;=\; e^{\,n} \;+\; 2$ and define \, $\mu
(\,x_{\,n}\,) \;=\; 1 \;,\; \mu (\,x_{\,n} \;+\; 1\,) \;=\;
\frac{1}{(\,e^{\,n} \;+\; 1\,)^{\,2}}$ \;\; for all \; $n
\;\,\geq\;\, 1$ \,,\, $\mu (\,x_{\,n} \;+\; x_{\,m}\,) \;\,=\;\, \mu
(\,x_{\,n} \;-\; x_{\,m}\,) \;\,=\;\, \mu (\,x_{\,n} \; x_{\,m}\,)
\;\,=\;\, \mu (\,x_{\,n} \; x_{\,m}^{\,- \,1}\,) \;\,=\;\, 1$ \,,\,
$\mu (\,x\,) \;=\; 0$ \; for all other \, $x \;\,\varepsilon\;\,
\mathbb{R}$. \, Here , \, $x_{\,n} \;\,\mathop  \to \limits_\mu
\;\,1 $ \, as \, $n \;\,\longrightarrow\;\, \infty$ \, but
$\{\,x_{\,n}\,\}_{\,n}$ \, is not \, $\mu\,-$ bounded in
$\mathbb{R_{\,\mu}}$.
\end{remark}
\begin{theorem}
Let $x_{\,n} \;\,\longrightarrow\;\, l$ \, and \, $y_{\,n}
\;\,\longrightarrow\;\, m$ \, as \, $n \;\,\longrightarrow\;\,
\infty$. Then \, $x_{\,n} \;+\; y_{\,n} \;\,\mathop  \to \limits_\mu
\;\, l \;+\; m $ \, as \, $n \;\,\longrightarrow\;\, \infty$
\end{theorem}
\begin{proof}
Obvious
\end{proof}
\begin{note}
$x_{\,n} \;\,\mathop  \to \limits_\mu \;\,l $ , $y_{\,n} \;\,\mathop
\to \limits_\mu \;\,m $ \, as \, $n \;\,\longrightarrow\;\, \infty$
\, do not always imply \, $x_{\,n} \;+\; y_{\,n} \;\,\mathop  \to
\limits_\mu \;\, l \;+\; m $ \, as \, $n \;\,\longrightarrow\;\,
\infty$. For example, let \, $x_{\,n} \;=\; y_{\,n} \;=\; \frac{(\,n
\;+\; 1\,)^{\,2}}{n^{\,2}}$ , \, $\mu (\,x_{\,n} \;-\; 1\,) \;=\;
\frac{n^{\,2}}{(\,2\,n \;+\; 1\,)^{\,3}}$, \, $\mu (\,2\,x_{\,n}\,)
\;=\; \frac{n^{\,2}}{2\,(\,n \;+\; 1\,)^{\,3}}$ , \, $\mu (\,t_{\,n}
\;+\; t_{\,m}\,) \;\,=\;\, \mu (\,t_{\,n} \;-\; t_{\,m}\,) \;\,=\;\,
\mu (\,t_{\,n} \; t_{\,m}\,) \;\,=\;\, \mu (\,t_{\,n} \;
t_{\,m}^{\,- \,1}\,) \;\,=\;\, 1$ , \, where \, $t_{\,n} \;=\;
x_{\,n} \;-\; 1$ , $m \;\neq\; n$ , \, $\mu (\,x\,) \;=\; 0$ \; for
all other \, $x \;\,\varepsilon\;\, \mathbb{R}$. \, Here we see that
\, $x_{\,n} \;\,\mathop  \to \limits_\mu \;\,1 $  , \, $y_{\,n}
\;\,\mathop  \to \limits_\mu \;\,1 $ \, as \, $n
\;\,\longrightarrow\;\, \infty$. \, But \, $x_{\,n} \;+\; y_{\,n}
\;\,\mathop  \to \limits_\mu \;\, 0 $ \, as \, $n
\;\,\longrightarrow\;\, \infty$.
\end{note}
\medskip
\begin{theorem}
If $inf\,\{\,\mu (\,x\,) \;\,: \;\, x \;\,\varepsilon\;\, \mathbb{R}
\,\}$ \; $= \;\; a \;\,>\;\, 0$ , \, $x_{\,n} \;\,\mathop  \to
\limits_\mu \;\,l $ , $y_{\,n} \;\,\mathop \to \limits_\mu \;\,m $
\, as \, $n \;\,\longrightarrow\;\, \infty$ \, then \, $x_{\,n}
\;+\; y_{\,n} \;\,\mathop  \to \limits_\mu \;\, l \;+\; m $ \, as \,
$n \;\,\longrightarrow\;\, \infty$.
\end{theorem}
\medskip
\begin{theorem}
Let $x_{\,n} \;\,\longrightarrow\;\, l$ \, and \, $y_{\,n}
\;\,\longrightarrow\;\, m$ \, as \, $n \;\,\longrightarrow\;\,
\infty$. Then \, $x_{\,n}\; y_{\,n} \;\,\mathop  \to \limits_\mu
\;\, l\, m $ \, as \, $n \;\,\longrightarrow\;\, \infty$
\end{theorem}
\medskip
\begin{note}
$x_{\,n} \;\,\mathop  \to \limits_\mu \;\,l $ , $y_{\,n} \;\,\mathop
\to \limits_\mu \;\,m $ \, as \, $n \;\,\longrightarrow\;\, \infty$
\, do not always imply \, $x_{\,n} \; y_{\,n} \;\,\mathop  \to
\limits_\mu \;\, l \, m $ \, as \, $n \;\,\longrightarrow\;\,
\infty$. For example, let \, $x_{\,n} \;=\; (\,1 \;+\;
\frac{1}{n}\,)^{\,2}$ \,and\, $\mu (\,x_{\,n} \;-\; 1\,) \;=\;
\frac{n^{\,2}}{(\,2\,n \;+\; 1\,)^{\,3}}\,;$ \, $y_{\,n} \;=\;
\frac{n \;+\; 1}{3\,n \;+\; 1}$ \,and\, $\mu (\,y_{\,n} \;-\;
\frac{1}{3}\,) \;=\; \frac{3\,(\,3\,n \;+\; 1\,)}{2\,n^{\,2}} \, ;$
\, $\mu (\,x_{\,n}\;y_{\,n}\,) \;=\; \frac{1}{n}$ \,and define $\mu$
suitably at all other points of $\mathbb{R}$ such that all the
conditions of the definition$(\,1\,)$ are being satisfied. Here we
see that \, $x_{\,n} \;\,\mathop  \to \limits_\mu \;\,1 $  , \,
$y_{\,n} \;\,\mathop  \to \limits_\mu \;\,\frac{1}{3} $ \, as \, $n
\;\,\longrightarrow\;\, \infty$. \, But \, $x_{\,n} \; y_{\,n}
\;\,\mathop  \to \limits_\mu \;\, 0 $ \, as \, $n
\;\,\longrightarrow\;\, \infty$.
\end{note}
\medskip
\begin{theorem}
If $inf\,\{\,\mu (\,x\,) \;\,: \;\, x \;\,\varepsilon\;\, \mathbb{R}
\,\}$ \; $= \;\; a \;\,>\;\, 0$ , \, $x_{\,n} \;\,\mathop  \to
\limits_\mu \;\,l $ , $y_{\,n} \;\,\mathop \to \limits_\mu \;\,m $
\, as \, $n \;\,\longrightarrow\;\, \infty$ \, then \, $x_{\,n} \;
y_{\,n} \;\,\mathop  \to \limits_\mu \;\, l \, m $ \, as \, $n
\;\,\longrightarrow\;\, \infty$.
\end{theorem}
\medskip
\begin{definition}
A sequence \, $\{\,x_{\,n}\,\}_{\,n}$ \, in $\mathbb{R_{\,\mu}}$ \,
is said to be $\mu\,-$ increasing if \, $x_{\,1} \;\leq_{\,\mu}\;
x_{\,2} \;\leq_{\,\mu}\;\; \cdots \;\;\leq_{\,\mu}\;\; x_{\,n}
\;\;\leq_{\,\mu}\;\; x_{\,n \;+\; 1} \;\;\leq_{\,\mu}\;\; \cdots$
\end{definition}
\medskip
\begin{theorem}
If \, $\{\,x_{\,n}\,\}_{\,n}$ \, is \, $\mu\,-$ increasing and
$\mu\,-$ bounded above then \, $\{\,x_{\,n}\;\mu
(\,x_{\,n}\,)\,\}_{\,n}$ converges to its supremum .
\end{theorem}
\medskip
\S \, 3 \hspace{1.5cm} $\mu\,-$ \textbf{Fuzzy Complex Numbers}
$\textbf{:}$
\begin{definition}
Let \, $z \;\,\varepsilon\;\, \mathbb{C_{\,\mu}}$ . $\mu\,-$
conjugate of $z$ is denoted by \, $\overline{z}_{\,\mu}$ and defined
by \, $\overline{z}_{\,\mu} \;=\; \overline{z}\,\mu (\,z\,)$ , \,
where \, $\overline{z}$ \, is the usual complex conjugate of $z$ .
\end{definition}
\medskip
\begin{property}
$(\,i\,)$ \hspace{1.5cm} $\overline{(\, \overline{z}_{\,\mu}
\,)}_{\,\mu} \;\,=\;\, z\,\,\mu (\,z\,)\,\,\mu
(\,\overline{z}_{\,\mu} \,)$ \\\\ $(\,ii\,)$ \hspace{1.5cm}
$\frac{\overline{(\,z_{\,1} \;+\; z_{\,2}\,)}_{\,\mu}}{\mu
(\,z_{\,1} \;+\; z_{\,2}\,)} \;\;=\;\; \frac{\overline{(\,z_{\,1}
\,)}_{\,\mu}}{\mu (\,z_{\,1}\,)} \;+\; \frac{\overline{(\,
z_{\,2}\,)}_{\,\mu}}{\mu (\, z_{\,2}\,)}$  , \; $\;
\mu\,(\,z_{\,1}\,) ,  \mu\,(\,z_{\,2}\,) , \mu\,(\,z_{\,1} \;+\;
z_{\,2}\,) \; \, \neq \; \, 0$ \\\\ $(\,iii\,)$ \hspace{1.5cm}
$\frac{\overline{(\,z_{\,1} \;-\; z_{\,2}\,)}_{\,\mu}}{\mu
(\,z_{\,1} \;-\; z_{\,2}\,)} \;\;=\;\; \frac{\overline{(\,z_{\,1}
\,)}_{\,\mu}}{\mu (\,z_{\,1}\,)} \;-\; \frac{\overline{(\,
z_{\,2}\,)}_{\,\mu}}{\mu (\, z_{\,2}\,)}$ , \; $\;
\mu\,(\,z_{\,1}\,) ,  \mu\,(\,z_{\,2}\,) , \mu\,(\,z_{\,1} \;-\;
z_{\,2}\,) \; \, \neq \; \, 0$ \\\\
$(\,iv\,)$ \hspace{1.5cm} $\frac{\overline{(\,z_{\,1} \;
z_{\,2}\,)}_{\,\mu}}{\mu (\,z_{\,1} \; z_{\,2}\,)} \;\;=\;\;
\frac{\overline{(\,z_{\,1} \,)}_{\,\mu}}{\mu (\,z_{\,1}\,)} \;\;
\frac{\overline{(\, z_{\,2}\,)}_{\,\mu}}{\mu (\, z_{\,2}\,)}$ , \;
$\; \mu\,(\,z_{\,1}\,) ,  \mu\,(\,z_{\,2}\,) , \mu\,(\,z_{\,1} \;
z_{\,2}\,) \; \, \neq \; \, 0$ \\\\
$(\,v\,)$ \hspace{1.5cm} $\frac{\overline{(\,z_{\,1} \;/\;
z_{\,2}\,)}_{\,\mu}}{\mu (\,z_{\,1} \;/\; z_{\,2}\,)} \;\;=\;\;
\frac{\overline{(\,z_{\,1} \,)}_{\,\mu}}{\overline{(\,
z_{\,2}\,)}_{\,\mu}} \;\; \frac{\mu (\, z_{\,2}\,)}{\mu (\,
z_{\,1}\,)}$ , \; $\; \mu\,(\,z_{\,1}\,) ,  \mu\,(\,z_{\,2}\,) ,
\mu\,(\,z_{\,1} \;/\; z_{\,2}\,) \; \, \neq \; \, 0$ \\\\ $(\,vi\,)$
\hspace{1.5cm} $z\,\mu (\,z\,) \;+\;
\overline{z}_{\,\mu} \;\,=\;\, 2\,Re (\,z\,)\;\mu (\,z\,)$ \\\\
$(\,vii\,)$ \hspace{1.5cm} $z\,\mu (\,z\,) \;-\;
\overline{z}_{\,\mu} \;\,=\;\, 2\,Im (\,z\,)\;\mu (\,z\,)$
\end{property}
\medskip
\begin{definition}
Let \, $z \;\,\varepsilon\;\, \mathbb{C_{\,\mu}}$ . $\mu\,-$ modulus
of $z$ is denoted by \, $|\,z\,|_{\,\mu}$ and defined by \,
$|\,z\,|_{\,\mu} \;=\; |\,z\,|\,\mu (\,z\,)$ , \, where \, $|\,z\,|$
\, is the usual modulus of $z$.
\end{definition}
\medskip
\begin{property}
$(\,i\,)$ \hspace{1.0cm} $\frac{|\,z_{\,1} \;
z_{\,2}\,|_{\,\mu}}{\mu (\,z_{\,1} \; z_{\,2}\,)} \;\;=\;\;
\frac{|\,z_{\,1} \,|_{\,\mu}}{\mu (\,z_{\,1}\,)} \;\; \frac{|\,
z_{\,2}\,|_{\,\mu}}{\mu (\, z_{\,2}\,)}$ , \; $\; \mu\,(\,z_{\,1}\,)
,  \mu\,(\,z_{\,2}\,) , \mu\,(\,z_{\,1} \; z_{\,2}\,) \; \, \neq \;
\, 0$ \\\\ $(\,ii\,)$ \hspace{1.5cm} $\frac{|\,z_{\,1} \;+\;
z_{\,2}\,|_{\,\mu}}{\mu (\,z_{\,1} \;+\; z_{\,2}\,)} \;\;\leq\;\;
\frac{|\,z_{\,1} \,|_{\,\mu}}{\mu (\,z_{\,1}\,)} \;\;+\;\; \frac{|\,
z_{\,2}\,|_{\,\mu}}{\mu (\, z_{\,2}\,)}$ , \; $\; \mu\,(\,z_{\,1}\,)
,  \mu\,(\,z_{\,2}\,) , \mu\,(\,z_{\,1} \;+\; z_{\,2}\,) \; \, \neq
\; \, 0$ \\\\ $(\,iii\,)$ \hspace{1.5cm} $\frac{|\,z_{\,1} \;/\;
z_{\,2}\,|_{\,\mu}}{\mu (\,z_{\,1} \;/\; z_{\,2}\,)} \;\;=\;\;
\frac{|\,z_{\,1} \,|_{\,\mu}}{|\, z_{\,2}\,|_{\,\mu}} \;\; \frac{\mu
(\, z_{\,2}\,)}{\mu (\, z_{\,1}\,)}$ , \; $\; \mu\,(\,z_{\,1}\,) ,
\mu\,(\,z_{\,2}\,) , \mu\,(\,z_{\,1} \;/\; z_{\,2}\,) \; \, \neq \;
\, 0$ \\\\ $(\,iv\,)$ \hspace{1.5cm}
$|\,\overline{z}_{\,\mu}\,| \;=\; |\,z\,|\,\mu (\,z\,)$ \\\\
$(\,v\,)$ \hspace{1.5cm} $\frac{|\,z_{\,1} \;-\;
z_{\,2}\,|_{\,\mu}}{\mu (\,z_{\,1} \;-\; z_{\,2}\,)} \;\;\geq\;\;
\frac{|\,z_{\,1} \,|_{\,\mu}}{\mu (\,z_{\,1}\,)} \;\;-\;\; \frac{|\,
z_{\,2}\,|_{\,\mu}}{\mu (\, z_{\,2}\,)}$ , \; $\; \mu\,(\,z_{\,1}\,)
,  \mu\,(\,z_{\,2}\,) , \mu\,(\,z_{\,1} \;-\; z_{\,2}\,) \; \, \neq
\; \, 0$ \\\\ $(\,vi\,)$ \hspace{1.5cm} $|\,z\,|_{\,\mu}
\;\,\geq\;\, Re (\,z\,)\;\mu (\,z\,)$ \,and\, $|\,z\,|_{\,\mu}
\;\,\geq\;\, Im (\,z\,)\;\mu (\,z\,)$ \\\\ $(\,vii\,)$
\hspace{1.5cm} $z\;\overline{z}_{\,\mu} \;\,=\,\;
z\;\overline{z}\,\mu (\,z\,) \;\,=\;\, |\,z\,|^{\,2}\;\mu (\,z\,)$
\end{property}
\medskip
\begin{definition}
Let \, $z \;\,\varepsilon\;\, \mathbb{C_{\,\mu}}$ . $\mu\,-$
argument of $z$ is denoted by \, $\arg_{\,\mu} (\,z\,)$ \, and
defined by \, $\arg_{\,\mu} (\,z\,) \;\,=\;\, \arg
(\,z\,)\,\mu(\,z\,)$ , \; $-\,\pi \;<\; \arg (\,z\,) \;\leq\; \pi$ .
\end{definition}
\medskip
\begin{property}
$\frac{\arg_{\,\mu}
(\,z_{\,1}\,z_{\,2}\,)}{\mu\,(\,z_{\,1}\,z_{\,2}\,)} \;\;=\;\;
\frac{\arg_{\,\mu} (\,z_{\,1}\,)}{\mu\,(\,z_{\,1}\,)} \;+\;
\frac{\arg_{\,\mu} (\,z_{\,2}\,)}{\mu\,(\,z_{\,2}\,)} \;+\;
2\,k\,\pi$ , where \, $k \;=\; 0$ \, if \, $-\,\pi \;<\;
\arg\,(\,z_{\,1}\,) \;+\; \arg\,(\,z_{\,2}\,) \;\leq\; \pi$ ; \, $k
\;=\; 1$ \, if \, $\arg\,(\,z_{\,1}\,) \;+\; \arg\,(\,z_{\,2}\,)
\;\leq\; -\,\pi$ ; \, $k \;=\; -\,1$ \, if \, $\arg\,(\,z_{\,1}\,)
\;+\; \arg\,(\,z_{\,2}\,) \;>\; \pi$  , \; $\; \mu\,(\,z_{\,1}\,)
\,,\, \mu\,(\,z_{\,2}\,) \,,\, \mu\,(\,z_{\,1} \; z_{\,2}\,) \; \,
\neq \; \, 0$ .
\end{property}
\medskip
\begin{definition}
Let \, $z \;\,\varepsilon\;\, \mathbb{C_{\,\mu}}$ . Then \,
$\exp_{\,\mu} (\,z\,) \;\,=\;\, \exp (\,z\,)\;\mu (\,z\,)$ , where
\, $\exp (\,z\,)$ \, denotes the usual exponential function in \,
$\mathbb{C}$ .
\end{definition}
\medskip
\begin{property}
$(\,i\,)$ \hspace{0.7cm} $\frac{\exp_{\,\mu} (\,z_{\,1} \;+\;
z_{\,2}\,)}{\mu (\,z_{\,1} \;+\; z_{\,2}\,)} \;\;\,=\;\;\,
\frac{\exp_{\,\mu} (\,z_{\,1} \,)}{\mu (\,z_{\,1}\,)} \;\,\;
\frac{\exp_{\,\mu} (\, z_{\,2}\,)}{\mu (\, z_{\,2}\,)}$ \; provided
\; $\; \mu\,(\,z_{\,1}\,) ,  \mu\,(\,z_{\,2}\,) ,\\ \mu\,(\,z_{\,1}
\;+\; z_{\,2}\,) \;\; \neq \;\;  0$ \\\\
$(\,ii\,)$ \hspace{1.5cm} $\frac{\exp_{\,\mu} (\,z_{\,1} \;-\;
z_{\,2}\,)}{\mu (\,z_{\,1} \;-\; z_{\,2}\,)} \;\;\,=\;\;\,
\frac{\exp_{\,\mu} (\,z_{\,1} \,)}{\exp_{\,\mu} (\, z_{\,2}\,)}
\;\,\; \frac{\mu (\, z_{\,2}\,)}{\mu (\, z_{\,1}\,)}$ , \; $\;
\mu\,(\,z_{\,1}\,) ,  \mu\,(\,z_{\,2}\,) , \mu\,(\,z_{\,1} \;-\;
z_{\,2}\,) \; \, \neq \; \, 0$
\end{property}
\medskip
\begin{note}
$\exp_{\,\mu} (\,0\,) \;\,=\;\, \exp (\,0\,)\;\mu (\,0\,) \;\,=\;\,
1$
\end{note}
\medskip
\begin{note}
$(\;\frac{exp_{\,\mu} (\,z\,)}{\mu (\,z\,)}\;)^{\,n} \;\;=\;\;
\frac{exp_{\,\mu} (\,n\,z\,)}{\mu (\,n\,z\,)}$ , \; $\mu\,(\,z\,)
\,,\, \mu\,(\,n\,z\,) \;\neq\; 0$
\end{note}
\medskip
\begin{definition}
Let \, $z \;\,\varepsilon\;\, \mathbb{C_{\,\mu}}$ \, and \, $z
\;\neq\; 0$. $\mu\,-$ logarithm of $z$ is denoted by \, $Log_{\,\mu}
(\,z\,)$ \, and defined by \, $Log_{\,\mu} (\,z\,) \;\,=\;\,
Log\,(\,z\,)\,\mu(\,z\,)$ .
\end{definition}
\medskip
\begin{property}
$(\,i\,)$ \hspace{0.7cm}
$\frac{Log_{\,\mu}\,(\,z_{\,1}\,z_{\,2}\,)}{\mu\,(\,z_{\,1}\,z_{\,2}\,)}
\;\;=\;\; \frac{Log_{\,\mu}\,(\,z_{\,1}\,)}{\mu\,(\,z_{\,1}\,)}
\;+\; \frac{Log_{\,\mu}\,(\,z_{\,2}\,)}{\mu\,(\,z_{\,2}\,)}$  , \;
$\; \mu\,(\,z_{\,1}\,) \, , \,  \mu\,(\,z_{\,2}\,) , \\
\mu\,(\,z_{\,1} \;
z_{\,2}\,) \; \, \neq \; \, 0$ \\\\
$(\,ii\,)$ \hspace{1.0cm}
$\frac{Log_{\,\mu}\,(\,z_{\,1}\,/\,z_{\,2}\,)}{\mu\,(\,z_{\,1}\,/\,z_{\,2}\,)}
\;\;=\;\; \frac{Log_{\,\mu}\,(\,z_{\,1}\,)}{\mu\,(\,z_{\,1}\,)}
\;-\; \frac{Log_{\,\mu}\,(\,z_{\,2}\,)}{\mu\,(\,z_{\,2}\,)}$ , \;
$\; \mu\,(\,z_{\,1}\,) ,  \mu\,(\,z_{\,2}\,) , \mu\,(\,z_{\,1} \;/\;
z_{\,2}\,) \; \, \neq \; \, 0$
\end{property}
\medskip
\begin{definition}
Let \, $a \;\,\varepsilon\;\, \mathbb{C_{\,\mu}}$ \,,\, $a \;\neq\;
0$ \, and \, $z \;\,\varepsilon\;\, \mathbb{C_{\,\mu}}$. We now
define \, $a_{\,\mu}^{\,z}$ \, by \, $a_{\,\mu}^{\,z} \;\,=\;\,
a^{\,z}\;\mu\,(\,z\,) \;\,=\;\, \exp_{\,\mu}\,(\,z\,Log\,a\,)$ \,
and \, $p.v.(\,a_{\,\mu}^{\,z}\,) \;\;=\;\;
p.v.(\,a^{\,z}\,)\,\mu\,(\,z\,)$
\end{definition}
\medskip
\begin{property}
$(\,i\,)$ \hspace{1.5cm} $\frac{p.v.(\,a_{\,\mu}^{\,z_{\,1} \;+\;
z_{\,2}}\,)}{\mu\,(\,z_{\,1} \;+\; z_{\,2}\,)} \;\;=\;\;
\frac{p.v.(\,a_{\,\mu}^{\,z_{\,1}\,}\,)}{\mu\,(\,z_{\,1}\,)} \;+\;
\frac{p.v.(\,a_{\,\mu}^{\,z_{\,2}}\,)}{\mu\,(\,z_{\,2}\,)}$ , \; $\;
\mu\,(\,z_{\,1}\,) \, , \,  \mu\,(\,z_{\,2}\,) , \\\mu\,(\,z_{\,1}
\;+\;
z_{\,2}\,) \; \, \neq \; \, 0$ \\\\
$(\,ii\,)$ \hspace{1.5cm} $(\,a\,b\,)_{\,\mu}^{\,z}\,\mu\,(\,z\,)
\;\,=\;\, a_{\,\mu}^{\,z}\,b_{\,\mu}^{\,z}$
\end{property}
\medskip
\bigskip
\textbf{Conclusion \; :} \; In next research, we are trying to
established the concepts of a function in \, $\mathbb{R_{ \,\mu} }$
\, or \, $\mathbb{C_{ \,\mu} }$ , limit , continuity ,
differentiability etc. of such function . \\\\

\textbf{Reference \; :} \\\\ $[\,1\,]$ \hspace{0.5cm} A. Kauffman \,
$\&$ \, M. Gupta , \textit{Introduction to Fuzzy Arithmetic} , Van
Nostrand Rrinhold,
New York, 1991 . \\
$[\,2\,]$ \hspace{0.5cm} S. Nanda , \textit{Fuzzy Field and Fuzzy
Linear Space} , Fuzzy Sets and Systems 19 $(\,1986\,)$ 89 - 94. \\
$[\,3\,]$ \hspace{0.5cm} Zhang Guang - Quan , \textit{Fuzzy
Continuous Function and its Properties} , Fuzzy Sets and Systems 43
$(\,1991\,)$ 159 - 171. \\ $[\,4\,]$ \hspace{0.5cm} Congxin Wu and
Jiqing Qiu , \textit{Some remarks for fuzzy complex analysis} ,
Fuzzy Sets and Systems 106 $(\,1999\,)$ 231 - 238. \\ $[\,5\,]$
\hspace{0.5cm} J. J. Buckley , \textit{Fuzzy Complex Numbers} ,
Fuzzy Sets and Systems 33 $(\,1989\,)$ 333 - 345.\\\\

\end{document}